\documentclass[journal]{IEEEtran}
\usepackage{stmaryrd}
\usepackage{amsfonts}
\usepackage{graphicx,times,amsmath}
\usepackage{graphics,color}
\usepackage{graphicx}
\usepackage{latexsym}
\usepackage{amsmath}
\usepackage{amsfonts}
\usepackage{amssymb}
\usepackage{url}
\usepackage{subfigure}
\usepackage{times}
\usepackage{color}
\newtheorem{theorem}{Theorem}
\newtheorem{corollary}{Corollary}
\newtheorem{definition}{Definition}
\newtheorem{lemma}{Lemma}
\newtheorem{remark}{Remark}

\newtheorem{assumption}{Assumption}

\newtheorem{proposition}{Proposition}
\allowdisplaybreaks

\begin{document}

\title{ $\mu$-Stability of Nonlinear Positive Systems With Unbounded Time-Varying Delays}

\author{Tianping~Chen,~{\it{Senior~Member,~IEEE,}}
        and~Xiwei~Liu,~{\it{Member,~IEEE}}
\thanks{Tianping Chen is with the School of Computer Sciences/Mathematical
Sciences, Fudan University, 200433, Shanghai, P.R. China. E-mail:
tchen@fudan.edu.cn}
\thanks{Xiwei Liu is with Department of Computer Science and
Technology, Tongji University, and with the Key Laboratory of Embedded System and Service Computing,
Ministry of Education, Shanghai 200092, China. E-mail:
xwliu@tongji.edu.cn}
\thanks{This work was supported by
the National Science Foundation of China under Grant No. 61273211, 61203149, 61233016, the National Basic Research Program of China (973 Program) under
Grant No. 2010CB328101, ``Chen Guang'' project supported by Shanghai
Municipal Education Commission and Shanghai Education Development
Foundation under Grant No. 11CG22, the Fundamental Research Funds
for the Central Universities, and the Program
for Young Excellent Talents in Tongji University.}
}

% make the title area
\maketitle

\begin{abstract}
\boldmath {\bf Stability of the zero solution plays an important role in the investigation of positive systems. In this note, we revisit the $\mu$-stability of positive nonlinear systems with unbounded time-varying delays. The system is modelled by continuous-time differential equations. Under some assumptions on the nonlinear functions like homogeneous, cooperative, nondecreasing, we propose a novel transform, by which the nonlinear system reduces to a new system. Thus, we analyze its dynamics, which can simplify the nonlinear homogenous functions with respect to (w.r.t.) arbitrary dilation map to those w.r.t. the standard dilation map. We finally get some criteria for the global $\mu$-stability. A numerical example is given to demonstrate the
validity of obtained results.}
\end{abstract}


\section{Introduction}
In many cases of the real world, states of systems should be always non-negative, for example, population levels in biology \cite{population}, industrial processes involving chemical reactors \cite{chemical}, transport and accumulation phenomena of substances in compartmental systems \cite{JS1993}-\cite{HC2004}. All these examples can be described mathematically by the model of \emph{positive systems} \cite{L1979}-\cite{FR2000}.

In the literature of positive systems, the zero solution takes an important position, for example, the zero solution means the death of all species in biosphere. Moreover, although many important and interesting properties have been reported and analyzed, the most fundamental one is the stability property, which has been comprehensively studied by researchers from different fields.

For a linear system to be positive $\dot{x}(t)=Ax(t)$, which is called positive linear time-invariant (LTI) system, a necessary and sufficient condition is that $A$ should be a Metzler matrix (i.e., off-diagonal entries are nonnegative). For positive linear time-delay systems, \cite{LA1980a}-\cite{LA1980b} show that the presence of constant delays does not affect the stability performance of the system; \cite{R2009}-\cite{LYW2010} (and \cite{NNSM2008}-\cite{NMNSN2009}) also report this fact for positive linear systems (and positive systems defined by functional and integro differential equations) of time-varying delays; \cite{N2013} present some criteria for exponential stability of positive LTI differential systems with distributed delay; \cite{N2006} gives an extension of the classical Perron-Frobenius theorem to positive quasi-polynomial matrices, then some necessary and sufficient conditions for the exponential stability of positive linear time-delay differential systems are obtained; \cite{LYW2009} addresses the asymptotic stability of discrete-time positive systems with bounded time-varying delays and proves that the stability is also determined by the delay-free systems; \cite{LL2013} establishes an equivalent relationship between asymptotical stability and exponential stability for discrete-time positive system for all bounded time-varying delays; \cite{FCJ2013} studies the asymptotic stability and decay rates with unbounded delays; and for positive linear switched system (PLS), \cite{MS2003}-\cite{FMC2009} analyze and prove that for 2-dimensional PLS, the necessary and sufficient condition for stability under arbitrary switching is that every matrix in the convex hull of the matrices defining the subsystems is Hurwitz, but is not true for $3$-dimensional PLS; and \cite{LD2011} also addresses the stability problem of both discrete-time and continuous-time PLSs with arbitrary (even unbounded) time delays.

However, all these above papers study only the linear systems, which is just a small fraction of the whole positive systems. Recently, some great progress on the homogeneous systems \cite{H1967} has been made, which is a particular class of nonlinear systems. To overcome this disadvantage, there have been some excellent works on nonlinear positive systems in the literature. For example, \cite{AL2002} is a pioneering work on generalizing the corresponding analysis from positive linear systems to homogeneous cooperative and irreducible systems; \cite{MV2009} shows that a constant-delayed homogeneous cooperative system is globally asymptotically stable (GAS) for all non-negative delays if and only if the system is GAS for zero delay; \cite{BMV2010} shows that GAS and cooperative systems, homogeneous of any order with respect to arbitrary dilation maps are D-stable with constant time delay; \cite{FCJ2014a} investigates the exponential stability of homogeneous positive systems of degree one with bounded time-varying delays; and \cite{FCJ2014b} generalizes the homogeneous positive systems to any degree, and the bounded time-varying delay to be unbounded, the asymptotic and $\mu$-stability are discussed for both continuous-time and discrete-time systems.

In all, there are few papers considering the stability with unbounded time delay, see \cite{LD2011} and \cite{FCJ2014b}. However, in our previous papers \cite{CW2007a}-\cite{CW2007b}, we are the first to give explicitly the definition of $\mu$-stability, which includes the exponential stability, the power-rate stability, the log stability, the log-log stability, etc. The key concept of defining a new Lyapunov function can even be retrieved in the first author's paper \cite{Chen2001}. Moreover, the proposed theory is successfully applied on the stability of equilibrium for neural networks \cite{LC2008}-\cite{WC2012}, synchronization \cite{CWZ08} and consensus \cite{LLC2010} for complex networks. Therefore, in this note, {we will revisit the $\mu$-stability of nonlinear positive systems with unbounded time-varying delays using our proposed method, by defining a simple transform of the variables, we can directly get the corresponding criteria, which are easier to read and understand than previous works \cite{CW2007b}.}

The rest of this note is organized as follows. In section \ref{pre}, some necessary definitions, lemmas and notations are given. In section \ref{transform}, a simple transform is introduced and three properties are obtained. In section \ref{main}, by transforming the original system to a simple form, whose nonlinear functions are homogeneous w.r.t. the standard dilation map, we analyze the $\mu$-stability. Moreover, a numerical example to show the process of transform and analysis by using the obtained theoretical results is carefully demonstrated in Section \ref{num}.
Finally, the note is concluded in section \ref{conclude}.

\section{Preliminaries}\label{pre}
Let $\mathbb{R} (\mathbb{R}_{+}), \mathbb{R}^n_{+}$ denote the field of real (positive) numbers and
the nonnegative orthant of all $n$-dimensional real space $\mathbb{R}^n$ respectively. If all elements of matrix A are non-negative (non-positive), then denote matrix $A\succeq 0 (\preceq 0)$.
A real $n\times n$ matrix $A=(a_{ij})$, denoted as $A\in R^{n\times n}$, is \emph{Metzler} if and only if its off-diagonal entries $a_{ij}, \forall i\ne j$ belong to $\mathbb{R}^{+}$. For two vectors $x$ and $y$ in $\mathbb{R}^n$, if $x_i\ge y_i, i=1,\cdots,n$, where $x_i$ ($y_i$) is used to denote the $i$-th component of $x$ ($y$), then we denote $x\ge y$.
%$A\succ 0 (\prec 0)$,  All elements of matrix A are positive (negative)
%$A^T$ Transpose of matrix A
%$\mathbb{R} $ The set of all real  numbers
%$\mathbb{M}$ The set of $n\times n$ Metzler matrices

\begin{definition}
A dynamical system with state space $\mathbb{R}^n$ is \emph{positive} if any trajectory of the system starting at an initial state in the positive orthant $\mathbb{R}^n_{+}$ remains forever in $\mathbb{R}^n_{+}$.
\end{definition}

The following definitions of cooperative functions, homogeneous functions and nondecreasing functions are widely used in the related works of investigating the stability of positive nonlinear systems with (or without time delays), like \cite{FCJ2014b}. Here we rewrite it to keep the self-integrity of this note.

\begin{definition}
A continuous vector field $f: \mathbb{R}^n\rightarrow \mathbb{R}^n$ which is continuously differentiable on $\mathbb{R}^n\backslash\{0\}$ is said to be \emph{cooperative} if the Jacobian matrix $\partial{f}/\partial{x}$ is Metzler for all $x\in \mathbb{R}^n_{+}\backslash\{0\}$.
\end{definition}

\begin{proposition} (see \cite{S1995})
Let $f: \mathbb{R}^n\rightarrow \mathbb{R}^n$ be cooperative. For any two vectors $x$ and $y$ in $\mathbb{R}^n_{+}$ with $x_i=y_i$ and $x\ge y$, we have $f_i(x)\ge f_i(y)$.
\end{proposition}

This Proposition can be obtained by the fact that for a cooperative function $f$, each $f_i(x)$ is a nondecreasing function of each $x_j$ for $j\ne i$.

%Let $x,y\in \mathbb{R}^n_{+}$, then $x\le y$ means $x_i\le y_i$, $\forall i=1,\cdots,n$. Furthermore, $x<y$ if and only if $x\le y$ and $x\ne y$, while $x\ll y$ if and only if $x_i<y_i$, $\forall i=1,\cdots,n$.

\begin{definition}\label{delta}
For an $n$-tuple $r=(r_1,\cdots,r_n)$ of positive real numbers and $\lambda>0$, the \emph{dilation map} $\delta_{\lambda}^{r}(x)$ is defined by
\begin{align*}
\delta_{\lambda}^{r}(x)=(\lambda^{r_1}x_1,\cdots,\lambda^{r_n}x_n).
\end{align*}
When $r=(1,1,\cdots,1)$, then dilation map is the \emph{standard dilation map}.
\end{definition}

\begin{definition}
A vector field $f(x), x\in \mathbb{R}^n$, is \emph{homogeneous of degree $p\ge 0$ w.r.t. the {dilation map} $\delta_{\lambda}^{r}(x)$} if
\begin{align}\label{wrt}
f(\delta_{\lambda}^{r}(x))=\lambda^{p}\delta_{\lambda}^{r}(f(x))
\end{align}
\end{definition}

\begin{definition}
A vector field $g: \mathbb{R}^n\rightarrow \mathbb{R}^n$ is said to be \emph{nondecreasing on $\mathbb{R}^n_{+}$}, if for any $x,y\in \mathbb{R}^n_{+}$, and $x\ge y$, then $g(x)\ge g(y)$.
\end{definition}

\section{A simple but useful transform}\label{transform}
In this section, we first define a simple transform as:
\begin{align}\label{tf}
z_i=x_i^{1/r_i},
\end{align}
where $r_i, i=1,\cdots,n$ is defined by the dilation map in Definition \ref{delta}. For any function $f$, we can define a new function $\overline{f}(z)$, where $z=(z_1,\cdots,z_n)$, as:
\begin{align}\label{nf}
\overline{f}_i(z)=\left\{\begin{array}{ll}\frac{f_i(z_1^{r_1},\cdots,z_m^{r_m})}{z_i^{r_i-1}}
=\frac{f_i(x)}{z_i^{r_i-1}},&z_i\ne0\\0,&z_i=0\end{array}\right.
\end{align}

Next, we will present three good properties for this new function $\overline{f}(z)$ and prove them in the form of lemmas.

\begin{lemma}\label{l1}
If $f(x), x\in \mathbb{R}^n$ is {homogeneous of degree $p\ge 0$ w.r.t. the {dilation map} $\delta_{\lambda}^{r}(x)$}, then $\overline{f}(z), z\in \mathbb{R}^n$ is homogeneous of degree $p$ w.r.t. the {standard dilation map}.
\end{lemma}
\begin{proof}
From the equation of (\ref{wrt}), one can get that:
\begin{align}
\overline{f}_i(\lambda z)&=\frac{f_i(\lambda^{r_1}z_1^{r_1},\cdots,\lambda^{r_n}z_n^{r_n})}
{\lambda^{r_i-1}z_i^{r_i-1}}\nonumber\\
&=\frac{f_i(\lambda^{r_1}x_1,\cdots,\lambda^{r_n}x_n)}
{\lambda^{r_i-1}z_i^{r_i-1}}\nonumber\\
&=\frac{f_i(\delta_{\lambda}^{r}(x))}{\lambda^{r_i-1}z_i^{r_i-1}}=\frac{\lambda^p \delta_{\lambda}^{r}f_i(x))}{\lambda^{r_i-1}z_i^{r_i-1}}\nonumber\\
&=\frac{\lambda^p \lambda^{r_i}f_i(x)}{\lambda^{r_i-1}z_i^{r_i-1}}=\lambda^p(\lambda\overline{f}_i( z))
\end{align}
where $\overline{f}(0)=0$. The proof is completed.
\end{proof}

\begin{lemma}\label{l2}
For the new function $\overline{f}(z)$ in (\ref{nf}), suppose the function $f$ is cooperative, then for any two vectors ${z}$ and ${w}$ in $\mathbb{R}^n_{+}$ with $z_i=w_i$ and ${z}\ge {w}$, we have $\overline{f}_i(z)\ge \overline{f}_i(w)$.
\end{lemma}
\begin{proof}
At first, from the definition of (\ref{nf}), if at least one of $z_i, w_i$ is zero, the above claim is proved; otherwise, we have
\begin{align}\label{rpc}
\overline{f}_i(z)=\frac{f_i(x)}{z_i^{r_i-1}},~~\overline{f}_i(w)=\frac{f_i(y)}{w_i^{r_i-1}},
\end{align}
where $z_i=x_i^{1/r_i}$ and $w_i=y_i^{1/r_i}$.

Therefore, if ${z}\ge {w}$ and $z_i=w_i$, then $x\ge y$ and $x_i=y_i$, according to Proposition 1, one can get that $f_i(x)\ge f_i(y)$ and $z_i^{r_i-1}=w_i^{r_i-1}$. Combined with the equation (\ref{rpc}), we have $\overline{f}_i(z)\ge \overline{f}_i(w)$.
\end{proof}

\begin{lemma}\label{l3}
For the new function $\overline{f}(z)$ in (\ref{nf}), suppose the function $f$ is nondecreasing and { $f_i(x)=\Omega(x_i)$, that is to say, there exists a function $d(x_1,\cdots,x_{i-1},x_{i+1},\cdots,x_n)>0$, such that $f_i(x)\ge d x_i$.} Then for any two vectors ${z}$ and ${w}$ in $\mathbb{R}^n_{+}$ with ${z}\ge {w}$, we have $\overline{f}_i(z)\ge \overline{f}_i(w)$.
\end{lemma}

\begin{proof}
In order to prove the conclusion, one can only need to check that $\partial{\overline{f}_i(z)}/\partial{x_j}$ is nonnegative.

At first, from the condition that $f$ is nondecreasing, one can get that
\begin{align*}
\frac{\partial{f_i(x)}}{\partial{x_j}}\ge 0,~~i,j=1,\cdots,n.
\end{align*}
From Definition (\ref{nf}), we have $\overline{f}_i(z)={f}_i(x)x_i^{(1-r_i)/r_i}$, therefore, if $i\ne j$
\begin{align*}
\frac{\partial{\overline{f}_i(z)}}{\partial{x_j}}
=\frac{\partial{f}_i(x)}{\partial{x_j}}x_i^{(1-r_i)/r_i}\ge 0;
\end{align*}
if $i=j\in\{1,\cdots,n\}$, we have
\begin{align*}
\frac{\partial{\overline{f}_i(z)}}{\partial{x_j}}
=&\frac{\partial{f}_i(x)}{\partial{x_j}}x_i^{(1-r_i)/r_i}+
\frac{1-r_i}{r_i}f_i(x)x_i^{1/r_i-2}\\
=&x_i^{1/r_i-2}\bigg[\frac{\partial{f}_i(x)}{\partial{x_j}}x_i+\frac{1-r_i}{r_i}f_i(x)\bigg]
\end{align*}
Therefore, one just need to prove that $\frac{\partial{f}_i(x)}{\partial{x_j}}x_i+\frac{1-r_i}{r_i}f_i(x)\ge 0$. If not, simple calculations show that
\begin{align*}
f_i(x)< d(x_1,\cdots,x_{i-1},x_{i+1},\cdots,x_n)\cdot x_i^{(r-1)/r}< dx_i,
\end{align*}
which contradicts to the conditions $f_i(x)\ge dx_i$.

Based on all above discussions, one can get that the function $\overline{f}_i$ is nondecreasing. The proof is completed.
\end{proof}

\section{Main Results}\label{main}
At first, we give the network model for continuous-time positive nonlinear systems with unbounded time-varying delay as:
\begin{align}\label{model}
\left\{
\begin{array}{rll}
\dot{x}(t)&=f(x(t))+g(x(t-\tau(t)),&t\ge 0\\
x(t)&=\varphi(t), &t\in (-\infty, 0]
\end{array} \right.
\end{align}
where $x_i, i=1,\cdots,n$ is the state variable, $\varphi(t)\in \mathcal{C}((-\infty,0],\mathbb{R}^n_{+})$ is the vector-valued function specifying the initial condition of the system. $\tau(t)$ is the time-varying delay with $\tau(t)<t$ but unbounded.

\begin{assumption}\label{a1}
Assume the following conditions hold:
\begin{enumerate}
  \item $f$ is cooperative and $g$ is nondecreasing on $\mathbb{R}^n_{+}$,
  \item $f$ and $g$ are homogeneous of degree $p$ w.r.t. the dilation map $\delta_{\lambda}^{r}(x)$.
\end{enumerate}

It isclear that these conditions are the generalization of linear positive systems, and $f(0)=g(0)=0$, so $0$ is a solution of system (\ref{model}).
\end{assumption}

Using the above transform, let $z_{i}^{r_{i}}(t)=x_{i}(t),$ then all
\begin{align}\label{new}
\dot{z}_{i}(t)=&\frac{1}{r_i}\cdot\frac{1}{{z}_{i}^{r_{i}-1}(t)}
[{f}_{i}({z}_{1}^{r_{1}}(t),\cdots,{z}_{n}^{r_{n}}(t))\nonumber
\\
&+{g}_i({z}_{1}^{r_{1}}(t-\tau(t)),\cdots,{z}_{n}^{r_{n}}(t-\tau(t)))]
\nonumber\\
=&\frac{1}{r_i}\cdot[\overline{f}_i(z(t))+\overline{g}_i(z(t-\tau(t)))]
\end{align}
where $\overline{f}_i(z(t))={f_i(x(t))}/{z_i^{r_i-1}(t)}$ and $\overline{g}_i(z(t-\tau(t)))={g_i(x(t-\tau(t)))}/{z_i^{r_i-1}(t-\tau(t))}$. From Lemma \ref{l1} to Lemma \ref{l3}, we have the following property of $\overline{f}_i$ and $\overline{g}_i$:
\begin{enumerate}
  \item $\overline{f}$ is cooperative and $\overline{g}$ is nondecreasing on $\mathbb{R}^n_{+}$,
  \item $\overline{f}$ and $\overline{g}$ are homogeneous of degree $p$ w.r.t. the standard dilation map.
\end{enumerate}

In order to describe the convergence of the zero solution, we first define a norm as:
\begin{align}\label{cr0}
\|z(t)\|_{\{\xi,+\infty\}}=\max_{i=1,\cdots,n}\xi_i^{-1}z_i(t)
\end{align}
Obviously, if $\|z(t)\|_{\{\xi,+\infty\}}\rightarrow 0$, then $z(t)\rightarrow 0$ and $x(t)\rightarrow 0$, as $t\rightarrow 0$.

\begin{remark}
In \cite{FCJ2014b}, a similar definition of weighted $l_{\infty}$ norm is defined by a vector $v>0$ as
\begin{align*}
\|x\|_{\infty}^{v}=\max_{1\le i\le n}\frac{|x_i|}{v_i}
\end{align*}
It is clear that  this norm has been introduced and used in discussing stability of neural networks successfully in \cite{CW2007b}, or even much earlier work \cite{Chen2001} and others.
\end{remark}

\begin{definition}
Suppose that $\mu(t):\mathbb{R}_{+}\rightarrow \mathbb{R}_{+}$ is nondecreasing function satisfying $\mu(t)\rightarrow +\infty$ as $t\rightarrow +\infty$. The positive nonlinear system (\ref{model}) is said to achieve the global $\mu$-stability, if there exist constant scalars $M>0$ and $T>0$ (which may be related to initial values), such that for all $t\ge T$
\begin{align}
\|z(t)\|_{\{\xi,+\infty\}}\le \frac{M}{\mu(t)},~~i=1,\cdots,n.
\end{align}
\end{definition}

\begin{remark}
The concept of $\mu$-stability was first proposed
in \cite{CW2007b} and for power-rate convergence in  \cite{CW2007a} to discuss stability of dynamical systems with unbounded delays. It pointed out that between asymptotical stability and exponential stability, there exist various convergence concepts. It also shows that in case of bounded delay, the convergence rate is exponentially. And as the delay is increasing, the convergence rate is decreasing from exponential to asymptotical.

\end{remark}

\begin{theorem}\label{th1}
Suppose Assumption \ref{a1} holds. If there exists a positive and nondecreasing function $\mu(t)$ satisfying $\lim\limits_{t\rightarrow +\infty}\mu(t)=+\infty$, and for $j=1,\cdots,n$
\begin{align}\label{cri}
\frac{1}{r_j}
\bigg[\frac{\overline{f}_j(\xi)}{\xi_j}+\bigg(\lim\limits_{t\rightarrow \infty}\frac{\mu(t)}{\mu(t-\tau(t))}\bigg)^{p+1}\frac{\overline{g}_j(\xi)}{\xi_j}\bigg]\nonumber\\
+\lim\limits_{t\rightarrow \infty}\frac{\dot{\mu}(t)}{\mu(t)^{1-p}}<0.
\end{align}
Then the zero solution of positive nonlinear system (\ref{new}) is globally $\mu$-stable, i.e., $z_j(t)=O(\mu^{-1}(t))$.
\end{theorem}

\begin{proof}
From the inequality (\ref{cri}), we can find a constant $T$ such that for all $t\ge T$,
\begin{align*}
\frac{1}{r_j}
\bigg[\frac{\overline{f}_j(\xi)}{\xi_j}+\bigg(\frac{\mu(t_0)}{\mu(t_0-\tau(t_0))}\bigg)^{p+1}\frac{\overline{g}_j(\xi)}{\xi_j}\bigg]+\frac{\dot{\mu}(t_0)}{\mu(t_0)^{1-p}}<0.
\end{align*}
Moreover, for any $\chi\ge 1$, the following inequality holds:
\begin{align}\label{cr2}
\chi\frac{1}{r_j}
\bigg[\frac{\overline{f}_j(\xi)}{\xi_j}+\bigg(\frac{\mu(t_0)}{\mu(t_0-\tau(t_0))}\bigg)^{p+1}\frac{\overline{g}_j(\xi)}{\xi_j}\bigg]+\frac{\dot{\mu}(t_0)}{\mu(t_0)^{1-p}}<0.
\end{align}

Define the Lyapunov function as:
\begin{align}
\mathcal{V}(t)=\max\{1,\sup_{s\le t}V(t)\},
\end{align}
where $V(t)=\mu(t)\cdot\|z(t)\|_{\{\xi,+\infty\}}$. Then $\mathcal{V}(t)$ is a nondecreasing function, and $V(t)\le \mathcal{V}(t)$. Now, we claim that $\mathcal{V}(t)$ is bounded. More precisely, we will prove that for all $t\ge T$, $\mathcal{V}(t)=\mathcal{V}(T)$.

In fact, for any time $t_0>T$, in case $V(t_0)<\mathcal{V}(t_0)$, then there exists a $\epsilon$-neighborhood of $t_0$, such that $V(t)<\mathcal{V}(t_0)$ for $t\in (t_0,t_0+\epsilon)$, i.e., $\mathcal{V}(t)$ is non-increasing at $t_0$.

Instead, if $V(t_0)=\mathcal{V}(t_0)$, there exists an index $j$ such that
\begin{align*}
\mu(t_0)\xi_j^{-1}z_j(t_0)=\mathcal{V}(t_0);
\mu(t_0)\xi_i^{-1}z_i(t_0)\le\mathcal{V}(t_0), i\ne j,
\end{align*}
which mean that
\begin{align*}
z_j(t_0)=\frac{V(t_0)}{\mu(t_0)}\xi_j; ~~\mathrm{and}~~ z_i(t_0)=\frac{V(t_0)}{\mu(t_0)}\xi_i, i\ne j.
\end{align*}

From the Assumption \ref{a1}, using the cooperative and homogeneous property of function $\overline{f}$, one can get that
\begin{align}\label{cr3}
\overline{f}_j(z(t_0))\le \overline{f}\bigg(\frac{V(t_0)}{\mu(t_0)}\xi_j\bigg)=\bigg(\frac{V(t_0)}{\mu(t_0)}\bigg)^{p+1}\overline{f}_j(\xi).
\end{align}

Moreover, since $V(t_0-\tau(t_0))\le \mathcal{V}(t_0)=V(t_0)$, therefore, $\mu(t_0-\tau(t_0))\xi_j^{-1}z_j(t_0-\tau(t_0))\le V(t_0)$, i.e.,
\begin{align*}
z_j(t_0-\tau(t_0))\le \frac{V(t_0)}{\mu(t_0-\tau(t_0))}\xi_j.
\end{align*}
From the Assumption \ref{a1}, using the nondecreasing and homogeneous property of function $\overline{g}$, one can get that
\begin{align}\label{cr4}
&\overline{g}_j(z(t_0-\tau(t_0)))\nonumber\\
\le& \overline{g}_j\bigg(\frac{V(t_0)}{\mu(t_0-\tau(t_0))}\xi_j\bigg)\le\bigg(\frac{V(t_0)}{\mu(t_0-\tau(t_0))}\bigg)^{p+1}\overline{g}_j(\xi).
\end{align}

Now, differentiating $V(t_0)$ by the upper-right Dini-derivative, and using the obtained results (\ref{cr2}), (\ref{cr3}) and (\ref{cr4}), we have
\begin{align*}
&D^{+}V(t)|_{t=t_0}=\lim\limits_{\Delta \rightarrow 0^{+}}\sup\frac{V(t_0+\Delta)-V(t_0)}{\Delta}|_{t=t_0}\nonumber\\
=&D^{+}[\mu(t)\max_{i=1,\cdots,n}\xi_i^{-1}z_i(t)]|_{t=t_0}\nonumber\\
=&D^{+}[\mu(t)\xi_j^{-1}z_j(t)]|_{t=t_0}\nonumber\\
=&\dot{\mu}(t_0)\xi_j^{-1}z_j(t_0)+\mu(t_0)\xi_j^{-1}\dot{z}_j(t_0)\nonumber\\
=&\dot{\mu}(t_0)\xi_j^{-1}z_j(t_0)+\mu(t_0)\frac{1}{\xi_jr_j}[\overline{f}_j(z(t_0))+\overline{g}_j(z(t_0-\tau(t_0)))]\nonumber\\
=&\frac{\dot{\mu}(t_0)}{\mu(t_0)}\mu(t_0)\xi_j^{-1}z_j(t_0)+\mu(t_0)\frac{1}{\xi_jr_j}
\bigg[\bigg(\frac{V(t_0)}{\mu(t_0)}\bigg)^{p+1}\overline{f}_j(\xi)\nonumber\\
&+\bigg(\frac{V(t_0)}{\mu(t_0-\tau(t_0))}\bigg)^{p+1}\overline{g}_j(\xi)\bigg]\nonumber\\
=&\frac{V(t_0)}{\mu(t_0)^p}\Bigg[\frac{V(t_0)^p}{r_j}
\bigg[\frac{\overline{f}_j(\xi)}{\xi_j}+\bigg(\frac{\mu(t_0)}{\mu(t_0-\tau(t_0))}\bigg)^{p+1}\frac{\overline{g}_j(\xi)}{\xi_j}\bigg]\nonumber\\
&+\frac{\dot{\mu}(t_0)}{\mu(t_0)^{1-p}}\Bigg]<0,
\end{align*}
therefore, $\mathcal{V}(t_0)$ is also non-increasing at $t_0$.

In summary, $\mathcal{V}(t)=\mathcal{V}(T)$ for all $t\ge T$, which implies that $\mu(t)\|z(t)\|_{\{\xi,+\infty\}}\le \mathcal{V}(T)$, i.e., the global $\mu$-stability can be realized.
The proof is completed.
\end{proof}

If we define the norm for $z(t)$ is as follows:
\begin{align}\label{cr10}
\|z(t)\|_{\{\xi,+\infty, r^{\star}\}}=\max_{i=1,\cdots,n}(\xi_i^{-1}z_i(t))^{r^{\star}},
\end{align}
then the corresponding results can be stated as:
\begin{corollary}
Suppose Assumption \ref{a1} holds, if there exists a positive and nondecreasing function $\mu(t)$ satisfying $\lim\limits_{t\rightarrow +\infty}\mu(t)=+\infty$, and for $j=1,\cdots,n$
\begin{align}\label{crp}
\frac{r^{\star}}{r_j}
\bigg[\frac{\overline{f}_j(\xi)}{\xi_j}+\bigg(\lim\limits_{t\rightarrow \infty}\frac{\mu(t)}{\mu(t-\tau(t))}\bigg)^{\frac{p+1}{r^{\star}}}\frac{\overline{g}_j(\xi)}{\xi_j}\bigg]\nonumber\\
+\lim\limits_{t\rightarrow \infty}\frac{\dot{\mu}(t)}{\mu(t)^{1-\frac{p}{r^{\star}}}}<0.
\end{align}
Then the zero solution of positive nonlinear system (\ref{new}) is globally $\mu$-stable, i.e., $z_j(t)=O(\mu^{-1/r^{\star}}(t))$.
\end{corollary}

The proof is similar to that in Theorem \ref{th1}. Here we omit the proof.

\begin{remark}
By adjusting the parameter $r^{\star}$, the condition (\ref{crp}) can be satisfied. In \cite{FCJ2014b}, $r^{\star}=\max_{i}r_i$; while in \cite{AL2002}, $r^{\star}\ge\max_{i}r_i$.
\end{remark}

In case $p=0$, the inequality (\ref{crp}) can be rewritten as:
\begin{align}\label{crq}
\frac{r^{\star}}{r_j}
\bigg[\frac{\overline{f}_j(\xi)}{\xi_j}+\bigg(\lim\limits_{t\rightarrow \infty}\frac{\mu(t)}{\mu(t-\tau(t))}\bigg)^{\frac{1}{r^{\star}}}\frac{\overline{g}_j(\xi)}{\xi_j}\bigg]
+\lim\limits_{t\rightarrow \infty}\frac{\dot{\mu}(t)}{\mu(t)}<0
\end{align}

Except for the exponential stability and power-rate stability, we give the Log-stability and Log-Log stability (firstly given in \cite{CW2007b}) criteria under this case.

\begin{corollary} (Log-stability)
Suppose Assumption \ref{a1} holds with $p=0$, time delay $\tau(t)\le t-t/\ln t$, if
\begin{align}\label{log}
{\overline{f}_j(\xi)}+{\overline{g}_j(\xi)}
<0,
\end{align}
Then the zero solution of positive nonlinear system (\ref{new}) is globally $\mathrm{Log}$-stable, i.e., $z_j(t)=O(\ln(t+1)^{-1})$.
\end{corollary}

\begin{proof}
Choose the function $\mu(t)=\ln(t+1)$, and simple calculations show that
\begin{align*}
\lim\limits_{t\rightarrow \infty}\frac{\mu(t)}{\mu(t/\ln t)}=1, ~~\mathrm{and}~~ \lim\limits_{t\rightarrow \infty}\frac{\dot{\mu}(t)}{\mu(t)}=0.
\end{align*}
Let $r^{\star}=1$, so the above inequality (\ref{crq}) is equivalent to the inequality of (\ref{log}). The proof is completed.
\end{proof}

\begin{corollary} (Log-Log stability)
Suppose Assumption \ref{a1} holds with $p=0$£¬ time delay $\tau(t)\le t-t^{\alpha}, 0<\alpha<1$, if inequality (\ref{log}) holds,
%\begin{align}\label{log-log}
%{\overline{f}_j(\xi)}+{\overline{g}_j(\xi)}
%<0,
%\end{align}
then the zero solution of positive nonlinear system (\ref{new}) is $\mathrm{Log}-\mathrm{Log}$ stable, i.e., $z_j(t)=O(\ln\ln(t+3)^{-1})$.
\end{corollary}

\begin{proof}
Choose the function $\mu(t)=\ln\ln(t+3)$, and simple calculations show that
\begin{align*}
\lim\limits_{t\rightarrow \infty}\frac{\mu(t)}{\mu(t^{\alpha})}=1, ~~\mathrm{and}~~ \lim\limits_{t\rightarrow \infty}\frac{\dot{\mu}(t)}{\mu(t)}=0.
\end{align*}
Let $r^{\star}=1$, so the above inequality (\ref{crq}) is equivalent to the inequality of (\ref{log}). The proof is completed.
\end{proof}

\section{Numerical examples}\label{num}
We choose the following example given by \cite{FCJ2014b}:
\begin{align}\label{add}
\dot{x}(t)=f(x(t))+g(x(t-\tau(t))
\end{align}
where
\begin{align*}
f(x_{1},x_{2})=\left[
\begin{array}{c}
-5x_{1}^{3}+2x_{1}x_{2} \\
x_1^{2}x_{2}-4x_{2}^{2}
\end{array}
\right],
\end{align*}
and
\begin{align*}
g(x_{1},x_{2})=\left[
\begin{array}{c}
x_{1}x_{2} \\
2x_{1}^{4}
\end{array}
\right].
\end{align*}

In this system, it is easy to check that $f$ is cooperative and $g$ is nondecreasing; moreover, $f$ and $g$ are homogeneous of degree $p=2$ w.r.t the dilation map $\delta_{\lambda}^r(x)$ with $r=(1,2)$.

The difference of this example with that in \cite{FCJ2014b} is that the time delay obeys the following condition:
\begin{align*}
\tau(t)=t-t/\ln t
\end{align*}

Let $z_1(t)=x_{1}(t), z_2(t)=x_{2}^{1/2}(t)$. Then, after this transformation, the system (\ref{add}) turns to be
\begin{align*}
\left\{
\begin{array}{ll}
\dot{z}_1(t)&=-5z_1^{3}+2z_1(t)z_2^{2}(t)+z_1(t-\tau(t))z_2^{2}(t-\tau(t))\\
\dot{z}_2(t)&=\frac{1}{2z_2(t)}[z_1^{2}(t)z_2^{2}(t)-4z_2^{4}(t)+2z_1^{4}(t-\tau(t))]
\end{array} \right.
\end{align*}
therefore,
\begin{align*}
\overline{f}(z_{1},z_{2})=\left[
\begin{array}{c}
-5z_1^{3}+2z_1(t)z_2^{2}(t) \\
(z_1^{2}(t)z_2^{2}(t)-4z_2^{4}(t))/z_2(t)
\end{array}
\right],
\end{align*}
and
\begin{align*}
\overline{g}(z_{1},z_{2})=\left[
\begin{array}{c}
z_1(t)z_2^{2}(t) \\
2z_1^{4}(t)/z_2(t)
\end{array}
\right].
\end{align*}
Obviously, these functions are homogeneous with degree $2$ w.r.t. the standard dilation map.

If we choose $\mu(t)=\ln(t+1)$, $\xi=(1,1)$ and $r^{\star}=2$, then
\begin{align}
\lim\limits_{t\rightarrow \infty}\frac{\mu(t)}{\mu(t/\ln t)}=\lim\limits_{t\rightarrow \infty}\frac{\ln(t+1)}{\ln(t/\ln t+1)}=1,\\
\lim\limits_{t\rightarrow \infty}\frac{\dot{\mu}(t)}{\mu(t)^{1-1}}=\lim\limits_{t\rightarrow \infty}\frac{1}{1+t}=0,
\end{align}
and for $j=1,2$, the left inequality of (\ref{crp}) in Corollary 1 is $-4$ and $-1$ respectively. Therefore, the zero solution of positive nonlinear system (\ref{add}) can achieve the $\mathrm{Log}$-stability, i.e., $z_j(t)=O(\ln(t+1)^{-1}), j=1,2$, or equivalently,
\begin{align}\label{logs}
x_1(t)=O(\ln(t+1)^{-1}),~~~x_2(t)=O(\ln(t+1)^{-2}).
\end{align}

Choose the initial values as $(x_1(t),x_2(t))=(1,4)$ for $t\in (-\infty,0]$, then Fig. \ref{nn_ac} shows that the trajectories of the positive nonlinear systems converge to zero. Moreover, even though the initial value of $x_1$ is larger than $x_2$, however, from our above analysis, the variable $x_2(t)$ converges more quickly than $x_1(t)$, Fig. \ref{nn_ac} also shows this phenomenon. Moreover, from the equation (\ref{logs}), one can get that a better description of Log-stability is to plot the trajectories of $\ln(x_i(t)), i=1,2$ w.r.t. $\ln(\ln(t+1))$, see Fig. 2, from which one can see that our theoretical result does be a good approximation of the real convergence rate.

\begin{figure}
\begin{center}
\includegraphics[width=0.5\textwidth,height=0.3\textheight]{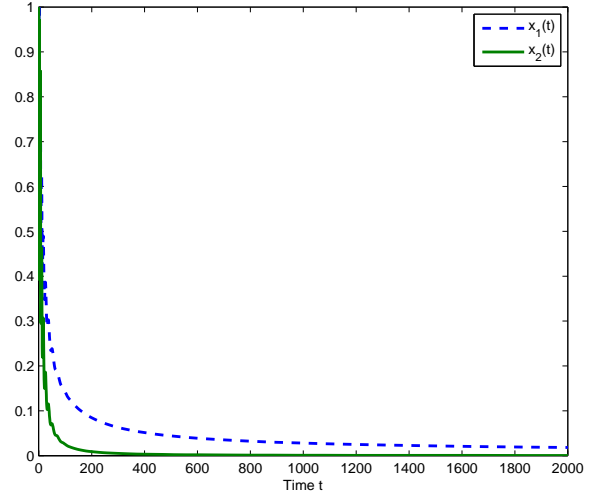}
\end{center}
\caption{Trajectories of $x_i(t), i=1, 2$ for positive nonlinear systems (\ref{add}) under unbounded time-varying delay $\tau(t)=t-t/\ln t$.} \label{nn_ac}
\end{figure}

\begin{figure}
\begin{center}
\includegraphics[width=0.5\textwidth,height=0.3\textheight]{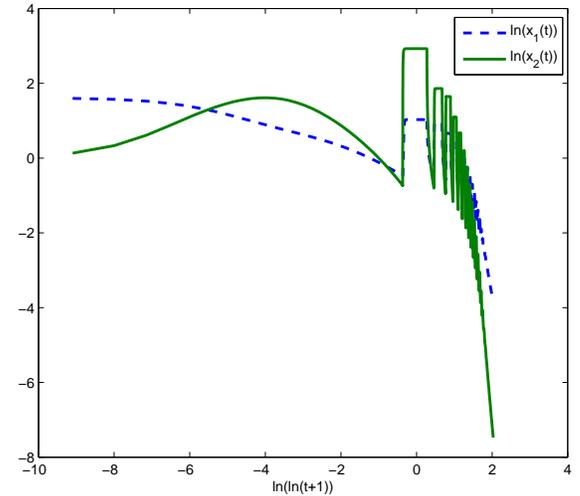}
\end{center}
\caption{Trajectories of $\ln(x_i(t)), i=1, 2$ w.r.t. $\ln(\ln(t+1))$} \label{logs}
\end{figure}

\section{Conclusions}\label{conclude}
In this note, the global $\mu$-stability of positive nonlinear systems with unbounded time-varying delays is revisited. We propose a novel transform, by which the positive nonlinear can be transformed to a system with homogeneous nonlinear functions w.r.t. the standard dilation map. Then, within this framework, we analyze various $\mu$-stability. Finally, a simple example is given to corroborate the effectiveness of the transform proposed in this note.

%
%\section*{Acknowledgement}
%The authors are very grateful to reviewers and editor for their
%valuable comments and suggestions to improve the presentation and
%theoretical results of this note. They also wish to thank an anonymous referee for calling their attention to event-triggered control.


\begin{thebibliography}{99}

\bibitem{population}
J. Hofbauer and K. Sigmund, \emph{Evolutionary Games and Population Dynamics}. Cambridge, UK: Cambridge University Press, 1998.

\bibitem{chemical}
F. Viel, F. Jadot, and G. Bastin, ``Global stabilization of exothermic chemical reactors under input constraints,'' Automatica, vol. 33, no. 8, pp. 1437-1448, Aug. 1997.

\bibitem{JS1993}
J.~A. Jacquez and C.~P. Simon, ``Qualitative theory of compartmental systems,'' \emph{SIAM Rev.}, vol. 35, no. 1, pp. 43-79, Mar. 1993.

\bibitem{H1998}
J.~M. Van Den Hof, ``Positive linear observers for linear compartmental systems,'' \emph{SIAM J. Control Optim.}, vol. 36, no. 2, pp. 590-608, Mar. 1998.

\bibitem{HC2004}
W.~M. Haddad and V.~S. Chellaboina, ``Stability theory for nonnegative and compartmental dynamical systems with time delay,'' \emph{Syst. Control Lett.}, vol. 51, no. 5, pp. 355-361, Apr. 2004.

\bibitem{L1979}
D.~G. Luenberger, \emph{Introduction to Dynamic Systems: Theory, Models and Applications}. New York: John Wiley, 1979.

\bibitem{FR2000}
L. Farina and S. Rinaldi, \emph{Positive Linear systems: Theory and Applications.} New York: Wiley, 2000.

\bibitem{LA1980a}
R.~M. Lewis and B.~D.~O. Anderson, ``Insensitivity of a class of nonlinear compartmental systems to the introduction of arbitrary time delay,'' \emph{IEEE Trans. Circuits Syst.}, vol. 27, no. 7, pp. 604-612, Jul. 1980.

\bibitem{LA1980b}
R.~M. Lewis and B.~D.~O. Anderson, ``Necessary and sufficient condition for delay-independent stability of linear autonomous systems,'' \emph{IEEE Trans. Autom. Control}, vol. 25, no. 4, pp. 735-739, Aug. 1980.

\bibitem{R2009}
M.~A. Rami, ``Stability anaysis and synthesis for linear positive systems with time-varying delays,'' in \emph{Proc. 3rd Multidisciplinary Intern. Symposium on Positive Systems: Theory and Applications}, 2009, pp. 205-215.

\bibitem{LYW2010}
X.~W. Liu, W.~S. Yu and L. Wang, ``Stability analysis for continuous-time positive systems with time-varying delays,'' \emph{IEEE Trans. Autom. Control}, vol. 55, no. 4, pp. 1024-1028, Apr. 2010.

\bibitem{NNSM2008}
P.~H.~A. Ngoc, T. Naito, J.~S. Shin, and S. Murakami, ``On stability and robust stability of positive linear Volterra equations,'' \emph{SIAM J. Control Optim.}, vol. 47, no. 2, pp. 975-996, 2008.

\bibitem{NMNSN2009}
P.~H.~A. Ngoc, S. Murakami, T. Naito, J.~S. Shin, and Y. Nagabuchi, ``On positive linear Volterra-Stieltjes differential systems,'' \emph{Integ. Equ. Oper. Theory}, vol. 64, no. 3, pp. 325-355, Jul. 2009.

\bibitem{N2013}
P.~H.~A. Ngoc, ``Stability of positive differential systems with delay,'' \emph{IEEE Trans. Autom. Control}, vol. 58, no. 1, pp. 203-209, Jan. 2013.

\bibitem{N2006}
P.~H.~A. Ngoc, ``A Perron-Frobenius theorem for a class of positive quasi-polynomial matrices,'' \emph{Appl. Math. Lett.}, vol. 19, no. 8, pp. 747-751, Aug. 2006.

\bibitem{LYW2009}
X.~W. Liu, W.~S. Yu, and L. Wang, ``Stability analysis of positive systems with bounded time-varying delays,'' \emph{IEEE Trans. Circuits Syst. II, Exp. Briefs}, vol. 56, no. 7, pp. 600-604, Jul. 2009.

\bibitem{LL2013}
X.~W. Liu and J. Lam, ``Relationships between asymptotic stability and exponential stability of positive delay systems,'' \emph{Int. J. Gen. Syst.}, vol. 42, no. 2, pp. 224-238, 2013.

\bibitem{FCJ2013}
H.~R. Feyzmahdavian, T. Charalambous, and M. Johansson, ``Asymptotic stability and decay rates of positive linear systems with unbounded delays,'' in \emph{Proc. 52nd IEEE Conf. Decision and Control}, 2013, pp. 6337-6342.

\bibitem{MS2003}
O. Mason and R. Shorten, ``A conjecture on the existence of common quadratic Lyapunov functions for positive linear systems,'' in \emph{Proc. Amer. Control Conf.}, 2003, pp. 4469-4470.

\bibitem{FMC2009}
L. Fainshil, M. Margaliot, and P. Chigansky, ``On the stability of positive linear switched systems under arbitrary switching laws,'' \emph{IEEE Trans. Autom. Control}, vol. 54, no. 4, pp. 897-899, Apr. 2009.

\bibitem{LD2011}
X.~W. Liu and C.~Y. Dang, ``Stability analysis of positive switched linear systems with delays,'' \emph{IEEE Trans. Autom. Control}, vol. 56, no. 7, pp. 1684-1690, Jul. 2011.
%
%\bibitem{ZLSL2014}
%X.~D. Zhao, X.~W. Liu, Y. Shen, and H.~Y. Li, ``Improved results on stability of continuous-time switched positive linear systems,'' \emph{Automatica}, vol. 50, no. 2, pp. 614-621, Feb. 2014.

\bibitem{H1967}
W. Hahn, \emph{Stability of Motion}. Berlin: Springer-Verlag, 1967.

%\bibitem{H1964}
%P. Hartman, \emph{Ordinary Differential Equations}. New York: John Wiley, 1964.

\bibitem{AL2002}
D. Aeyels and P. De Leenheer, ``Extension of the Perron-Frobenius theorem to homogeneous systems,'' \emph{SIAM J. Control Optim.}, vol. 41, no. 2, pp. 563-582, Jul. 2002.

\bibitem{MV2009}
O. Mason and M. Verwoerd, ``Observations on the stability properties of cooperative systems,'' \emph{Syst. Control Lett.}, vol. 58, no. 6, pp. 461-467, Jun. 2009.

\bibitem{BMV2010}
V.~S. Bokharaie, O. Mason and M. Verwoerd, ``D-stability and delay-independent stability of homogeneous cooperative systems,'' \emph{IEEE Trans. Autom. Control}, vol. 55, no. 12, pp. 2882-2885, Dec. 2010.

\bibitem{FCJ2014a}
H.~R. Feyzmahdavian, T. Charalambous, and M. Johansson, ``Exponential stability of homogeneous positive systems of degree one with time-varying delays,'' \emph{IEEE Trans. Autom. Control}, vol. 59, no. 6, pp. 1594-1599, Jun. 2014.

\bibitem{FCJ2014b}
H.~R. Feyzmahdavian, T. Charalambous, and M. Johansson, ``Asymptotic stability and decay rates of homogeneous positive systems with bounded and unbounded delays,'' \emph{SIAM J. Control Optim.}, vol. 52, no. 4, pp. 2623-2650, 2014.

\bibitem{CW2007a}
T.~P. Chen and L.~L. Wang, ``Power-rate global stability of dynamical systems with unbounded time-varying delays,'' \emph{IEEE Trans. Circuits Syst. II, Exp. Briefs}, vol. 54, no. 8, pp. 705-709, Aug. 2007.

\bibitem{CW2007b}
T.~P. Chen and L.~L. Wang, ``Global $\mu$-stability of delayed neural networks with unbounded time-varying delays,'' \emph{IEEE Trans. Neural Netw.}, vol. 18, no. 6, pp. 1836-1840, Nov. 2007.

\bibitem{Chen2001}
T.~P. Chen, ``Global exponential stability of delayed hopfield neural networks,'' \emph{Neural Netw.}, vol. 14, no. 8, pp. 977-980, Oct. 2001.

\bibitem{LC2008}
X.~W. Liu and T.~P. Chen, ``Robust $\mu$-stability for uncertain stochastic neural networks with unbounded time-varying delays,'' \emph{Physica A}, vol. 387, no. 12, pp. 2952-2962, May 2008.

\bibitem{LLC2011}
B. Liu, W.~L. Lu, and T.~P. Chen, ``Generalized halanay inequalities and their applications to neural networks with unbounded time-varying delays,'' \emph{IEEE Trans. Neural Netw.}, vol. 22, no. 9, pp. 1508-1513, Sep. 2011.

\bibitem{WC2012}
L.~L. Wang and T.~P. Chen, ``Complete stability of cellular neural networks with unbounded time-varying delays,'' \emph{Neural Netw.}, vol. 36, pp. 11-17, Dec. 2012.

\bibitem{CWZ08}
T.~P. Chen, W. Wu, and W.~J. Zhou, ``Global $\mu$-synchronization of linearly coupled unbounded time-varying delayed neural networks with unbounded delayed coupling,'' \emph{IEEE Trans. Neural Netw.}, vol. 19, no. 10, pp. 1809-1816, Oct. 2008.

\bibitem{LLC2010}
X.~W. Liu, W.~L. Lu, and T.~P. Chen, ``Consensus of multi-agent systems with unbounded time-varying delays,'' \emph{IEEE Trans. Autom. Control}, vol. 55, no. 10, pp. 2396-2401, Oct. 2010.

\bibitem{S1995}
H. Smith, \emph{Monotone Dynamical Systems: An Introduction to the Theoty of Competitive and Cooperative Systems.} Providence, RI: AMS, 1995.
\end{thebibliography}
\end{document}